\documentclass[times]{my}

\begin{document}

\title{A Mathematical Model of Population Growth as a Queuing System}
\shorttitle{A Mathematical Model of Population Growth as a Queuing System}

\author{Mariia Nosova\affil{1}}
\abbrevauthor{M. Nosova}
\headabbrevauthor{Nosova, M}

\address{%
\affilnum{1}Tomsk University of Control Systems and Radioelectronics, Tomsk, Russia}

\correspdetails{nosovamgm@gmail.com}

\received{}
\revised{}
\accepted{}

\communicated{ }

\begin{abstract}

In this article, a new mathematical model of human population growth as an autonomous
non-Markov queuing system with an unlimited number of servers and two 
types of applications is proposed. The research of this system was carried out a 
virtual phase method and a modified method of asymptotic analysis of a 
stochastic age-specific density for a number of applications 
served in the system at time t and was proofed that the asymptotic 
distribution is Gaussian. The main probabilistic characteristics of
this distribution are found. The mathematical model and methods 
of its research can be applied to analysis of
the population growth both in a single country and around the world.
\end{abstract}

\maketitle

\section{Introduction}

Human society is a complex system that is
constantly evolving and changing. Recently, researchers are
particularly attracted to the field of demography. On the one hand,
there is a constant increase in the worlds
population, on the other hand, in a number of countries the population
is decreasing. To develop an effective demographic policy in the modern
economy, the analysis and forecasting of the processes of reproduction
of the size and structure of the population are of particular
relevance.

Population projections are used in various fields of state and regional management, in marketing research and
insurance. Estimates of the future demographic situation are extremely
useful in the development of any socioeconomic development programs,
and in planning for the energy sector, transport production, and
services. The significance of this kind of research is also dictated by
the fact that the aggravation of the
demographic situation is the result of serious economic and social
changes that have occurred in society over the past several decades.
Demographic changes affect all areas of the economy. In order to manage
reproductive processes and solve global problems of human development,
modeling of demographic processes is necessary.

\section{Mathematical Methods and Models in Demography: A Review} 

Many research works are devoted to the issues
of modeling demographic processes. The result of these studies is an
extensive scientific material focused on the study and forecasting of
demographic processes, both in a single country and around the world.

For demographic forecasting of the size and
composition of the population, statistical methods and mathematical
modeling methods are used. The use of mathematical models refers to the
modern and most effective method of demographic analysis and consists
of determining some initial assumptions regarding the dynamics of
demographic processes.

Depending on whether the demographic model
takes into account the deviation of the frequencies of demographic
events from their probabilities, demographic models are usually divided
into stochastic (or probabilistic) and deterministic. Also, demographic
models are divided into continuous and discrete.

The first mathematical models of reproduction
include models of reproduction of biological populations. The simplest
model known as Fibonacci is the problem of breeding a single pair of
rabbits, which was described in 1202 by Leonardo Pisano (about 1170 --
about 1250). 

The next time a model of the total population
in demography is used, it is the model of the English scientist Thomas
Robert Malthus (1766 -- 1834) \cite{bibid23}. In the Malthus model, the
population has exponential growth. Note that the Malthus population and
Fibonacci population models are deterministic continuous models, in
particular, discrete models.

As for the simple deterministic models of
human growth, these include, first of all, models of linear and
exponential growth  \cite{bibid16, bibid17, bibid38}.  In the linear growth model, a change in
population occurs with constant absolute growth. It should be noted
that the linear growth model gives satisfactory results only for a
short period, while extension for a longer period neither in the past
nor in the future gives adequate results.

A stationary model of population reproduction
is also found in the literature \cite{bibid10}. In such a model, it is assumed
that the birth rate is equal to mortality, that is, with constant
population size, the growth is zero. Thus, in the stationary
population, the total number and number in the age-sex groups remain
constant. 

The best-known approach to describing the
population reproduction model as a whole is the theory of a stable
population \cite{bibid14, bibid25, bibid26, bibid30}. A stable population is characterized by
constant age-related intensities of fertility, mortality, and age
structure of the population. A population with an invariable survival
function is called the exponential population. It corresponds to all
the basic formulas of the theory of a stable population, which do not
include the fertility function. It is obvious that the stationary
population is a special case of a stable population.

Continuous and discrete analogs of the stable
population model are known. Continuous models are based on the integral
equation of population reproduction (Lotka-Volterra equation), while
discrete models are based on the matrix model (Leslie matrix).

The idea of exponential
population growth as the basis of a stable population was first
introduced by the English actuary M. Hale. The development of the
theory of a stable population is associated with such names as L. Euler
(1707 -- 1783), G. Knapp (1842--1926), V. Lexis (1837 -- 1914), J.
Lotka (1880 -- 1949), V. Bortkevich (1868 -- 1931). Note that J. Lotka
developed a theory of a stable population in continuous form for one
gender. The discrete theory of a stable population, using the discrete
approach of W. Feller (1906 -- 1970), was constructed by P. Leslie
\cite{bibid22}. A generalization of the theory of a stable population was carried
out by P. Vincent, J. Hajnal (1924 -- 2008),
P. Carmel, K. Gini (1884 -- 1965), J. Bourgeois-Pichat (1912 -- 1990)
\cite{bibid4}, L. Anry, N. Keyfitz \cite{bibid18, bibid19}, A. Ya. Boyarsky (1906 -- 1985), as
well as H. Hyrenius (1914--1979), A. Rorrers \cite{bibid33}, and B F. Shukailo
(1932 -- 1983). The Russian demographers S. A. Novoselsky (1872 --
1953), V.V. Paevsky (1893 --1934), A. Ya. Boyarsky (1906 -- 1985), I.G.
Venetsky (1914--1981) and others \cite{bibid5, bibid39} made a great contribution to
the development of practical methods for the practical application of a
stable population.

The current stage of development of the theory
of a stable population is associated with a generalization of the main
conclusions for the case of demographic processes with variable
intensities (fertility and mortality processes in the form of time
series, random processes). These have been explored by A. Coale (1917
-- 2002) \cite{bibid8} and A. Lopes, as well as R. Lee \cite{bibid21}, D. Ahlburg \cite{bibid1}, Z.
Sykes \cite{bibid37}, Carter, M. Alho \cite{bibid2}, B. Spencer \cite{bibid3}, J. Pollard \cite{bibid29}, J.
Cohen \cite{bibid9}, N. Keyfitz \cite{bibid18, bibid19}, H. Caswell \cite{bibid6, bibid7}, M. Trevisan \cite{bibid7}, L.
Goodman \cite{bibid11} and others. 

At the end of the nineteenth century,
scientists made an attempt to use the logistic curve as a model
describing the growth of mankind \cite{bibid6, bibid7, bibid10, bibid16, bibid17, bibid25, bibid26}. R. Pearl
(1879 -- 1940) and L. Reed \cite{bibid31} adhered to this model in their
research. The equation of this curve was introduced for the first time
in 1835 by the Belgian mathematician F. Verhulst (1804 -- 1849), whose
general idea was to superimpose on the exponential population growth
some factor slowing down this growth and increasing its effect as the
population grew. It should be borne in mind that the use of R. Pearl
and other demographers of the logistic curve for abstract calculations
of the population meant ignoring the influence of socioeconomic factors
on the population. In this regard, attempts to use the logistic curve
to predict the population for a fairly distant future led to unreliable
and incorrect conclusions. The logistic curve can be useful only for a
short-term forecast of the population, as well as an approximation of
the dynamics of some demographic indicators, for example, fertility,
mortality, survival function, etc. \cite{bibid20}.

Another deterministic model of population
reproduction is the model of hyperbolic population change. The first
assumption that the population growth rate is proportional to the
square of the population was observed in 1960 by V. Foerster (1911 --
2002), P. Mora, and L. Amiot. According to the same formula, I. S.
Shklovsky predicted the worlds population from 1600 to
1960 \cite{bibid39}. At present, S. P. Kapitsa is working in this direction \cite{bibid15}.

A significant contribution to mathematical
demography was made by Russian scientist Oleg Staroverov (1933--2006)
\cite{bibid35, bibid36}, who considered demographic processes in the form of Markov
models in the form of Markov chains. O. Staroverov, applying the theory
of Poisson flows, suggested that all demographic processes, including
migration, have probabilistic properties: stationarity, lack of
aftereffect, ordinary. From the simplest assumptions, O. Staroverov
obtained models of demography, migration, inter-industry, and social
movement of the population. Models of the natural movement of the
population were studied in discrete and continuous form; for them, the
basic equations were obtained in \cite{bibid35, bibid36}.

In \cite{bibid35}, O. Staroverov proposed a stochastic
model of population development with discrete-time, which takes into
account randomness in both fertility and mortality. The essence of the
study using such a model is as follows: the concept of an
individuals state is introduced at a point in time,
determined by the age group in which he is located, and possible
transitions and probabilities of these transitions are determined. In
addition, the use of Markov chains is found in \cite{bibid12}.

In a separate group, we can distinguish models
that consider the demographic process of population change as a
branching process, in a particular case as a process of death and
reproduction \cite{bibid11, bibid29, bibid32}.

A special place in demographic modeling is
occupied by the method of moving ages \cite{bibid24, bibid30, bibid34, bibid40}. The component
method or the method of moving ages was developed by P.K. Whelpton
(1893 -- 1964) \cite{bibid40}. In this method, the population distribution by age
is taken as the basis and the numbers of
individual age groups are gradually shifted in accordance with the
indicators of mortality tables.

In Russia, prospective estimates of the
population by the component method have been done by S. G. Strumilin
(1877 -- 1974), A. Ya. Boyarsky (1906 -- 1985), P. P. Shusherin, and M.
S. Bedny, M. G. Nosova \cite{bibid28}.

It should be noted that the formulas for the
model of natural population movement in discrete time (matrix model --
Leslie matrix), also obtained by O. Staroverov \cite{bibid35, bibid36} are the
so-called moving ages, that is, the application of the Leslie approach
leads to the same results as the component method. And the Lotka model
is nothing but an analog of the component method in continuous time.

The general theory of a stable population in
the case of demographic processes with variable intensities, used by A.
Cole, A. Lopez, R. Leah, D. Ahlburg, Z. Sykes, Carter, M. Alho, B.
Spencer, J. Pollard, J. Cohen, N. Keyfitz, H. Caswell, M. Trevisan, L.
Goodman, and others, stand out here in a separate class of methods
-- a stochastic version of the component method \cite{bibid1, bibid2, bibid6, bibid9, bibid11, bibid18, bibid19,bibid21,bibid29,bibid34}.

Note that the component method inherently does
not take into account the influence of socioeconomic factors on
demographic processes. In addition, the main drawback of the method is
the impossibility of calculating the birth rate when designing forecast
values for the number of age groups.

To summarize, the analysis of many models
shows that in the modeling of demographic processes, deterministic
models (discrete and continuous) and stochastic discrete are most
common. The advantage of stochastic models over deterministic models is
that they take into account the deviation of the frequencies of
demographic events from their probabilities.

However, demographic processes proceed in
continuous time and are stochastic. Research methods for such processes
in mathematical demography are not sufficiently developed. That is why
the urgent task is to significantly expand
the mathematical models of the process of changing the demographic
situation, as well as the development of research methods.

\section{New Stochastic Model with
Continuous Time}

Models and methods of queuing theory are
widely used in mathematical modeling of systems in various fields, for
example, communications, transport, industry, economics, and military
affairs. The main elements of any queuing system are the incoming
stream of applications (or requirements), service devices, and the law
of service, which determine the length of time an application is on the
device during its service.

The article proposes to apply models and
methods of queuing theory to analyze the processes of changing the
demographic situation, taking into account the specifics of demographic
processes. We note that the terminology established in the queuing
theory (device, line, service, application, requirement) is weakly
associated with demographic concepts, while the terms flow and
intensity are quite acceptable here.

In this paper, consider a mathematical model
of human population growth as an autonomous non-Markov queuing system
with an unlimited number of servers and two types of applications. Such
a model is defined as the flow of applications in an autonomous
non-Markov queuing system with an unlimited number of devices and the
average number of occupied devices (Figure 1).

We define the process of functioning of the
queuing system in terms of queuing theory, and then give a demographic
interpretation of the main elements of this model.

\begin{figure}[h]
\centering
\includegraphics[width=8.493cm,height=6.165cm]{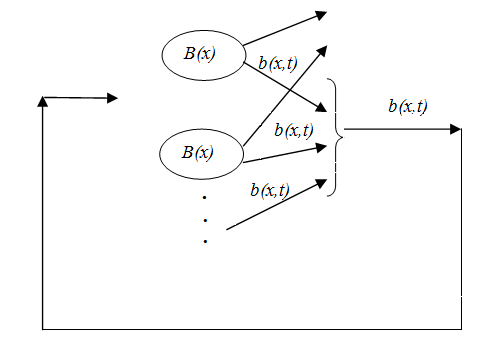}
\caption{Autonomous non-Markov queuing system
with an unlimited number of servers}
\label{figure: header text example}
\end{figure}

\newpage
The autonomous queuing system receives two types of applications: applications of the first type and applications of the second type. Define the process of servicing applications. Each application at the time of its receipt occupies a free device and is on it for the entire service time, the duration of which is random. Durations of servicing various requirements are stochastically independent, have the same distribution determined by the function \textit{S${}_{i}$}(\textit{x})=1-\textit{B${}_{i}$}(\textit{x}), where \textit{B${}_{1}$}(\textit{x}) and \textit{B${}_{2}$}(\textit{x}) are the distribution functions of the service time of applications of the first and second type, respectively. After completing the service, the application leaves the system.

For applications in the system, we define the age \textit{x}$\mathrm{\ge}$0 as the length of the interval from the time \textit{t}-\textit{x} of the beginning of its service (the moment of entry into the system) to the current time \textit{t}. Each application of the first type of age \textit{x} at time \textit{t} with intensity \textit{b}(\textit{x,t}) generates a new requirement, that is, the probability that the application of the first type of age \textit{x }from time \textit{t} for an infinitely small time interval of duration $\Delta$\textit{t} will generate a new requirement is \textit{b}(\textit{x,t})$\Delta$\textit{t + o}($\Delta$\textit{t}), and the probability of generating two or more requirements is an infinitely small quantity of a higher order than $\Delta$\textit{t}. A new application (of the first or second type) at the time of its appearance takes a free device and begins the process of its maintenance, generating the requirements of a new generation.

In the considered queuing system, there is no external source of applications, since all new applications are generated by applications that are being serviced. In this regard, the system is called autonomous. In the general case, it is assumed that the duration of the service time has a rather arbitrary distribution function; therefore, this service system is non-Markovian. The number of serviced applications is unlimited.

In terms of demography, the served application is interpreted as a female or male person, the application service time is the person's life expectancy, \textit{S${}_{1}$}(\textit{x}) and \textit{S${}_{2}$}(\textit{x}) are the survival function for women and men, respectively, the application age \textit{x} is the person's age at the considered time \textit{t}, function \textit{b}(\textit{x,t}) is the birth rate of women of age \textit{x }in year \textit{t} (fertility function). We assume that with a probability \textit{r} a girl is born and with a probability of (\textit{1}-\textit{r}) a boy. The incoming flow of applications is the process of the birth of children, that is, the sequence of moments of the birth of children from the entire population of women.

Note that in deterministic models with discrete-time, the age structure of the female population is considered as the aggregate of the number of the main five-year (or annual) demographic groups. In determinate models with continuous time, the age density of the female population is studied. In the proposed stochastic model with continuous time, the stochastic density of population and the probability distribution of its values are studied.

To determine the stochastic density, we denote \textit{N${}_{1}$}(\textit{x}${}_{1}$,\textit{x}${}_{2}$,\textit{t}) and \textit{N${}_{2}$}(\textit{x}${}_{1}$,\textit{x}${}_{2}$,\textit{t}) -- the number of applications of the first and second type with the age $x\in [x_1,x_2)$, served in this system at time \textit{t}. Here \textit{x} is any nonnegative real number. A limit 
\[\mathop{lim}_{\mathit{\Delta}x\to 0}\frac{1}{\mathit{\Delta}x}N_1(x,x+\mathit{\Delta}x,t)=\xi (x,t),\] 
\[\mathop{lim}_{\mathit{\Delta}x\to 0}\frac{1}{\mathit{\Delta}x}N_2\left(x,x+\mathit{\Delta}x,t\right)=\eta \left(x,t\right).\] 
will be called the stochastic densities $\xi$(\textit{x},\textit{t}) and $\eta$(\textit{x},\textit{t}) of the number of applications (female and male population) at age \textit{x} at time \textit{t}.

Obviously,
\[N_1\left(x_1,x_2,t\right)=\int^{x_2}_{x_1}{N_1\left(x,t\right)dx},\] 
\[N_2(x_1,x_2,t)=\int^{x_2}_{x_1}{N_2(x,t)dx}.\] 

  In this work, we will solve the problem of investigating the random functions $\xi$(\textit{x},\textit{t}) and $\eta$(\textit{x},\textit{t}). In demography the mathematical expectation of the random functions $\xi$(\textit{x},\textit{t}) and $\eta$(\textit{x},\textit{t}), there are
\[M\left\{\xi \left(x,t\right)\right\}=g\left(x,t\right),\] 
\[M\left\{\eta \left(x,t\right)\right\}=m\left(x,t\right).\] 
It is called a population density function at age \textit{x} at time \textit{t }\cite{bibid27}. The functions \textit{m}(\textit{x},\textit{t}) and \textit{g}(\textit{x},\textit{t}) define the average characteristics and they are the most important characteristics of the random functions $\xi$(\textit{x},\textit{t}) and $\eta$(\textit{x},\textit{t}). 

\section{Autonomous Non-Markov Queuing System with an Unlimited Number of Servers}

We perform an investigation of the random functions \textit{N${}_{1}$(x,t) }and \textit{N${}_{2}$(x,t) } according to the virtual phase method \cite{bibid27}. That is, we considerate an auxiliary autonomous queuing system with a PH-distribution of the service time (a phase-type distribution).

\subsection{A research of the auxiliary autonomous queuing system with a PH-distribution of the service time}

From the study of the autonomous queuing system defined by the Figure 1, we turn to the consideration of a system whose structure coincides with the original, and the service time $\tau$ of each application is composed of the durations of a random number of phases
\[\tau ={\tau }_1+{\tau }_2+...+{\tau }_{\nu },\] 
where $\tau$\textit{${}_{i}$} -- the duration of the\textit{ i}${}^{th}$ service phase. The values $\tau$\textit{${}_{i}$} are independent exponentially distributed random variables with the parameter $\mu$${}_{1\ }$for applications of the first type and with parameter $\mu$${}_{2\ }$for applications of the second type, where \textit{i}=1,2{\dots}\textit{v}. Here \textit{v} is a random variable and \textit{v}=1,2{\dots}. This system is called an autonomous system with phase distribution or PH distribution of service time.

Let us explain the service process using the example of applications of the first type. Service for each new application of the first type begins in the first phase. The application, having completed the service at the \textit{i}${}^{th}$ phase, with probability \textit{q${}_{i}$} proceeds to the service at the \textit{i} + 1${}^{th}$ phase, and with the probability 1-\textit{q${}_{i}$} completes the full service and leaves the system. 

We determine the transition probability of application to the next phase as
\begin{equation} \label{GrindEQ__1_} 
q_i=S_1\left(\frac{i}{{\mu }_1}\right)/S_1\left(\frac{i-1}{{\mu }_1}\right), 
\end{equation} 
we believe that \textit{S${}_{1}$}(0)=1-\textit{B${}_{1}$}(0)=1. So, we can write by the multiplication theorem on probability 
\[P\left(v=1\right)=P\left(v=1/v-1=0\right)P\left(v-1=0\right)=1-q_1,\] 
\[ P\left(v=1\right)=P\left(v=1/v-1=0\right)P\left(v-1=0\right)=1-q_1,\] 
\[P\left(v=2\right)=P\left(v=2/v-1=1\right)P\left(v-1=1\right)=r_1(1-q_2),\] 
\[ \cdots\] 
\[P\left(v=n\right)=P\left(v=n/v-1=n-1\right)P\left(v-1=n-1\right)=q_1q_2\dots q_{n-1}(1-q_n).\] 

According to this expression and \eqref{GrindEQ__1_} it follows 
\begin{equation} \label{GrindEQ__2_} 
P(\nu =n)=S_1\left(\frac{n-1}{{\mu }_1}\right)-S_1\left(\frac{n}{{\mu }_1}\right), 
\end{equation} 

where \textit{P}(\textit{v}=\textit{n}) is the probability that the application service will be completed in \textit{n} phase in the system. Since the service duration $\tau$ is a random non-negative variable and distributed exponentially with the parameter ${\mu }_1$, the characteristic function of the random variable $\tau$ has the form \textit{w(}$\alpha$)=\textit{Me}${}^{-}$${}^{\mathrm{\alpha }\mathrm{\tau }}$. We write by the formula of total probability for mathematical expectations and according to \eqref{GrindEQ__2_}
\[w(\alpha )=Me^{-\alpha \tau }=\sum^{\infty }_{n=1}{M(}e^{-\alpha \tau }/\nu =n)P(\nu =n)=\] 
\[={\sum^{\infty }_{n=1}{\left(\frac{{\mu }_1}{{\mu }_1+\alpha }\right)}}^n\left\{S_1\left(\frac{n-1}{{\mu }_1}\right)-S_1\left(\frac{n}{{\mu }_1}\right)\right\}.\] 
For ${\mu }_1$$\mathrm{\to}$$\mathrm{\infty}$, we rewrite this equation in the form
\[\mathop{lim}_{\mu \to \infty }Me^{-\alpha \tau }=\mathop{lim}_{\mu \to \infty }{\sum^{\infty }_{n=1}{\left[{\left(1-\frac{\alpha }{{\mu }_1+\alpha }\right)}^{{\mu }_1}\right]}}^{\frac{n}{{\mu }_1}}\left\{S_1\left(\frac{n-1}{{\mu }_1}\right)-S_1\left(\frac{n}{{\mu }_1}\right)\right\}=\int^{\infty }_0{e^{-\alpha x}dB_1(x)}.\] 

Thus, for ${\mu }_1$$\mathrm{\to}$$\mathrm{\infty}$, the duration of the service with the PH-distribution, by virtue of the choice of \eqref{GrindEQ__1_} probability values \textit{q${}_{i}$}, converges to the service time determined by the distribution function \textit{B${}_{1}$}(\textit{x})=1-\textit{S${}_{1}$}(\textit{x}) of the original autonomous queuing system.

Assuming that the application of the second type, having completed the service in the \textit{i}${}^{th}$ phase, with probability \textit{s${}_{i}$} goes to the service in the \textit{i} + 1${}^{st}$ phase, and with the probability 1-\textit{ s${}_{i\ }$}it completes the full service and leaves the system, it is not difficult to carry out similar reasoning. 

The method of approximating the application service time by the sum of a random number of independent and equally exponentially distributed random variables and the limit transition with an unlimited increase in the number of phases and a proportional decrease in the duration of each phase was called in \cite{bibid27} the virtual phase method.

We will assume that each application of the first type served at the\textit{ i}${}^{th}$ phase at time\textit{ t }with intensity \textit{b${}_{i}$(t)=b(i/$\mu$${}_{1}$,t)} generates a new application. We denote \textit{n}(\textit{i},\textit{t}) -- the number of applications of the first type and \textit{l}(\textit{i},\textit{t}) -- the number of applications of the second type, served at the appropriate\textit{ i}${}^{th}$ phase at time\textit{ t }. Then a random process
\[\overline{n}(t)=\{n\left(1,t\right),n\left(2,t\right),\dots ,\ l\left(1,t\right),l\left(2,t\right),\dots {\}}^T\] 
is multidimensional the continuous-time Markov chain. Its probability distribution is   
\[P(n_1,n_2,...l_1,l_2,\dots ,t)=P\left\{n_1\left(t\right)=n_1,...,l_1\left(t\right)=l_1,l_2\left(t\right)=l_2,\dots \right\}.\] 
We write the system of Kolmogorov differential equations  
\[\frac{\partial }{\partial t}\left\{P\left(n_1,n_2,\dots l_1,l_2,\dots ,t\right)\right\}=-P(n_1,n_2,...l_1,l_2,\dots ,t)\left\{\sum^{\infty }_{i=1}{n_i}\left({\mu }_1+b_i\left(t\right)\right)+l_i{\mu }_2\right\}+\] 
\[+P(n_1-1,n_2,...,l_1,l_2,\dots ,t)r\left\{(n_1-1)b_1\left(t\right)+\right.\sum^{\infty }_{i=2}{n_i}\left.b_i\left(t\right)\right\}+\] 
\[+P\left(n_1,n_2,\dots ,l_1-1,l_2,\dots ,t\right)\left(1-r\right)\sum^{\infty }_{i=2}{n_i}b_i\left(t\right)+\]                                           
\begin{equation} \label{GrindEQ__3_} 
+{\mu }_1\sum^{\infty }_{i=1}{(n_i+1)\{P(n_1,n_2,\dots ,n_i+1,}n_{i+1},\dots ,t)(1-q_i)+ 
\end{equation} 
\[+P(n_1,n_2,\dots ,n_i+1,n_{i+1}-1,n_{i+2}\dots ,t)q_i\}+\]                                     
\[+{\mu }_2\sum^{\infty }_{i=1}{(l_i+1)\{P(n_1,\dots ,l_i+1,}l_{i+1},\dots ,t)(1-s_i) +\] 
\[+P(n_1,\dots ,l_i+1,l_{i+1}-1,l_{i+2}\dots ,t)s_i\}.\] 
We denote a characteristic function of the number of occupied servers at time \textit{t} in the form
\[H\left(y,t\right)=M\left\{{exp \left(j\sum^{\infty }_{i=1}{(u_in_i(t)}+z_il_i\left(t\right))\right)\ }\right\} =\] 
\[=\sum_{n_1,n_2,\dots ,{,l}_1,l_2,\dots }{P\left(n_1,n_2,\dots {,l}_1,l_2,\dots ,t\right){exp \left\{j\sum^{\infty }_{i=1}{{(u}_in_i\left(t\right)+z_il_i\left(t\right))}\right\}\ }},\] 
 where $j=\sqrt{-1}$\textit{ -- }the imaginary unit\textit{.}

The characteristic function\textit{ H}(y,\textit{t}) is a multidimensional function of the vector argument \textit{y}=$\mathrm{\{}$\textit{u}${}_{1}$,\textit{u}${}_{2}$,{\dots},\textit{ z${}_{1}$, z${}_{2}$,{\dots}$\}$${}^{T}$} and the scalar argument \textit{t}. Multiply \eqref{GrindEQ__3_} by ${exp \left(j\sum^{\infty }_{i=1}{(u_in_i(t)}+z_il_i\left(t\right))\right)\ }$ and sum. From system \eqref{GrindEQ__3_} we obtain the equality
\[\frac{\partial H\left(y,t\right)}{\partial t}=j\sum^{\infty }_{i=1}{\frac{\partial H\left(y,t\right)}{\partial u_i}\Bigg\{ \left(1-re^{ju_1}\right)b_i\left(t\right)-\left(1-r\right){b_i\left(t\right)e}^{jz_1}+}\] 
\begin{equation} \label{GrindEQ__4_} 
+\left(1-e^{-ju_i}\right)\Bigg.{\mu }_1+e^{-ju_i}\left(1-e^{ju_{i-1}}\right){\mu }_1q_i\Bigg\}+ 
\end{equation} 
\[+j\sum^{\infty }_{i=1}{\frac{\partial H\left(y,t\right)}{\partial z_i}\left\{\bigg(1-e^{-jz_i}\right){\mu }_2-\bigg.}e^{-jz_i}\left(1-e^{jz_{i-1}}\right)\bigg.{\mu }_2s_i\bigg\},\] 
where $b_i(t)=b\left(\frac{i}{{\mu }_1},t\right)$, $q_i=S_1\left(\frac{i}{{\mu }_1}\right)/S_1\left(\frac{i-1}{{\mu }_1}\right)$,$\ s_i=S_2\left(\frac{i}{{\mu }_2}\right)/S_2\left(\frac{i-1}{{\mu }_2}\right).$

The solution \textit{H(y,t)} of \eqref{GrindEQ__4_} determines the problem of researching the autonomous system with the PH-distribution of the service time.

We denote
\[u_i=u(i/{\mu }_1),\ n_i(t)=n(i/{\mu }_1,t),\] 
\[z_i=u(i/{\mu }_2),\ l_i(t)=l(i/{\mu }_2,t),\] 
and the characteristic function \textit{H}(y,\textit{t}) is written as
\[H\left(y,t\right)=M\left\{\mathrm{exp}\mathrm{}\left[j\sum^{\infty }_{i=1}{\left(\frac{1}{{\mu }_1}u\left(\frac{i}{{\mu }_1}\right){\mu }_1n\left(\frac{i}{{\mu }_1},t\right)+\frac{1}{{\mu }_2}z\left(\frac{i}{{\mu }_2}\right){\mu }_2l\left(\frac{i}{{\mu }_2},t\right)\right)}\right]\right\}.\] 
Let an expression \textit{i}/$\mu$${}_{1}$$\mathrm{\to}$\textit{x }and\textit{ i}/$\mu$${}_{2}$$\mathrm{\to}$\textit{x }for $\mu$${}_{1}$$\mathrm{\to}$$\mathrm{\infty}$, $\mu$${}_{2}$$\mathrm{\to}$$\mathrm{\infty}$, \textit{i}$\mathrm{\to}$$\mathrm{\infty}$, then we assume that the following limits exist 
\[{\mathop{\mathrm{lim}}_{i/{\mu }_1\to x} u(i/{\mu }_1)\ }=u\left(x\right),\ {\mathop{\mathrm{lim}}_{i/{\mu }_1\to x} {\mu }_1n(i/{\mu }_1,t)\ }=\zeta \left(x,t\right),\] 
\[{\mathop{\mathrm{lim}}_{i/{\mu }_2\to x} z(i/{\mu }_2)\ }=z\left(x\right),\ {\mathop{\mathrm{lim}}_{i/{\mu }_2\to x} {\mu }_2l(i/{\mu }_2,t)\ }=\eta \left(x,t\right),\] 
\[{\mathop{\mathrm{lim}}_{ \begin{array}{c}
i/{\mu }_1\to x \\ 
i/{\mu }_2\to x \end{array}
} H(u,z,t)\ }=M\left\{\mathrm{exp}\mathrm{}\left[j\int^{\infty }_0{\left(u\left(x\right)\zeta \left(x,t\right)+z\left(x\right)\eta \left(x,t\right)\right)}dx\right]\right\}=F\left(y,t\right).\] 

The function \textit{F}(y,\textit{t}) is called the characteristic functional of the random functions $\xi$(\textit{x},\textit{t}) and $\eta$(\textit{x},\textit{t}) of two arguments \textit{x} and \textit{t}. The random function $\xi$(\textit{x},\textit{t}) is called the stochastic density of the number of applications of the first type of age \textit{x} served in the system at time \textit{t}, for all \textit{x}$\mathrm{\ge}$0, and the random function $\eta$(\textit{x},\textit{t}) is called the stochastic density of the number of applications of the second type of age \textit{x} served in the system at time \textit{t}, for all \textit{x}$\mathrm{\ge}$0.

With this in mind and for \textit{i}/$\mu$${}_{1}$$\mathrm{\to}$\textit{x, i}/$\mu$${}_{2}$$\mathrm{\to}$\textit{x,} $\mu$${}_{1}$$\mathrm{\to}$$\mathrm{\infty}$, $\mu$${}_{2}$$\mathrm{\to}$$\mathrm{\infty}$, we rewrite equation \eqref{GrindEQ__4_} in the form
\[\frac{\partial F\left(y,t\right)}{\partial t}=j\int^{\infty }_0{\frac{\partial F\left(y,t\right)}{\partial u\left(x\right)}}\Bigg\{\left(1-re^{ju\left(0\right)}\right)b\left(x,t\right)-\left(1-r\right)e^{jz\left(0\right)}b\left(x,t\right)\Bigg. +\] 
\begin{equation} \label{GrindEQ__5_} 
+\left(e^{-ju\left(x\right)}-1\right)\frac{S '_1\left(x\right)}{S_1\left(x\right)}-\Bigg.ju '\left(x\right)\Bigg\}+   
\end{equation} 
\[+j\int^{\infty }_0{\frac{\partial F\left(y,t\right)}{\partial z\left(x\right)}}\left\{\left(e^{-jz\left(x\right)}-1\right)\frac{S_2 '\left(x\right)}{S_2\left(x\right)}-jz '\left(x\right)\right\}dx.\] 

The equation \eqref{GrindEQ__5_} is called the basic equation for the autonomous non-Markov queuing system with an unlimited number of devices.

We will solve the equation \eqref{GrindEQ__5_} will be solved by the modified method of asymptotic analysis \cite{bibid27}, assuming that the random functions $\xi$(\textit{x},\textit{t}) and $\eta$(\textit{x},\textit{t}) take sufficiently large values proportional to some infinitely large value of \textit{N, }that is \textit{N} $\mathrm{\to}$$\mathrm{\infty}$. We find an asymptotic of the first and second orders of its solution for equation \eqref{GrindEQ__5_}.

\subsection{ Foundation of the probabilistic characteristics by the method of asymptotic analysis}

In queuing theory, the method of asymptotic analysis is the study of equations, determined any characteristics of a system, when an asymptotic condition is satisfied. The form of an asymptotic condition is specified in accordance with the model and the research problem [27]. 

We consider equation \eqref{GrindEQ__5_} for the characteristic functional \textit{F(y,t)} in the asymptotic condition of a large number of groups that are proportional to an infinitely large value of \textit{N}. In the asymptotic condition of a large number of groups, we find the characteristic function of the number of applications $\overline{n}(t)$, serviced at time \textit{t}.

\subsubsection{The first order asymptotic}

Denote $\varepsilon$=1/\textit{N}, where $\varepsilon$ -- a small positive parameter, in the equation \eqref{GrindEQ__5_} we replace 
\[u\left(x\right)=\varepsilon w_1\left(x\right),z\left(x\right)=\varepsilon w_2\left(x\right),F\left(y,t\right)=F_1\left(w_1,w_2,t,\varepsilon \right),\] 
and we obtain an equality
\[\frac{\partial F_1(w_1,w_2,t,\varepsilon )}{\partial t}=j\frac{1}{\varepsilon }\int^{\infty }_0{\frac{\partial F_1(w_1,w_2,t,\varepsilon )}{\partial w_1(x)}}\Bigg\{\left(1-re^{j\varepsilon w_1\left(0\right)}\right)b\left(x,t\right)\Bigg. -\] 
\begin{equation} \label{GrindEQ__6_} 
-\left(1-r\right)e^{j{\varepsilon w}_2\left(0\right)}b\left(x,t\right)+\left(e^{-j{\varepsilon w}_1\left(x\right)}-1\right)\frac{S_1 '\left(x\right)}{S_1\left(x\right)}-\Bigg.j\varepsilon {w_1} '\left(x\right)\Bigg\}+                     
\end{equation} 
\[+j\frac{1}{\varepsilon }\int^{\infty }_0{\frac{\partial F_1(w_1,w_2,t,\varepsilon )}{\partial w_2(x)}}\left\{\left(e^{-j{\varepsilon w}_2\left(x\right)}-1\right)\frac{S_2 '\left(x\right)}{S_2\left(x\right)}-j\varepsilon {w_2} '\left(x\right)\right\}dx.\] 
This equality \eqref{GrindEQ__6_} is singularly perturbed equation. 

We consider a subclass of solutions \textit{F(y},\textit{t}) of equation \eqref{GrindEQ__5_}, for which, due to replacements, the function $F_1\left(w_1,w_2,t,\varepsilon \right)\ $has the following properties.

There is a finite limit at $\varepsilon$$\mathrm{\to}$0 for $F_1\left(w_1,w_2,t,\varepsilon \right)$
\begin{equation} \label{GrindEQ__7_} 
{\mathop{\mathrm{lim}}_{\varepsilon \to 0} F_1\left(w_1,w_2,t,\varepsilon \right)=F_1\left(w_1,w_2,t\right),\ } 
\end{equation} 
and its derivative with respect \textit{t}, $w_1\left(x\right)$ and $w_2\left(x\right)$
\begin{equation} \label{GrindEQ__8_} 
{\mathop{\mathrm{lim}}_{\varepsilon \to 0} \frac{{\partial F}_1\left(w_1,w_2,t,\varepsilon \right)}{\partial t}=\frac{F_1\left(w_1,w_2,t\right)}{\partial t},\ } 
\end{equation} 
\begin{equation} \label{GrindEQ__9_} 
{\mathop{\mathrm{lim}}_{\varepsilon \to 0} \frac{{\partial F}_1\left(w_1,w_2,t,\varepsilon \right)}{\partial w_1(x)}=\frac{F_1\left(w_1,w_2,t\right)}{\partial w_1(x)}\ } ,                                                    
\end{equation} 
\begin{equation} \label{GrindEQ__10_} 
{\mathop{\mathrm{lim}}_{\varepsilon \to 0} \frac{{\partial F}_1\left(w_1,w_2,t,\varepsilon \right)}{\partial w_2(x)}=\frac{F_1\left(w_1,w_2,t\right)}{\partial w_2(x)}\ } .                                                  
\end{equation} 
We prove the following theorem.

\begin{theorem}
\label{thm1}
For $\varepsilon$$\mathrm{\to}$0 the solution of problem \eqref{GrindEQ__6_} has the form
\begin{equation} \label{GrindEQ__11_} 
F_1\left(w_1,w_2,t\right)=exp\left\{j\int^{\infty }_0{(w_1\left(x\right)g\left(x,t\right)+w_2\left(x\right)m\left(x,t\right))dx}\right\},\  
\end{equation} 
where \textit{g}(\textit{x},\textit{t}) and \textit{m}(\textit{x},\textit{t}) are solutions of equation 
\begin{equation} \label{GrindEQ__12_} 
\frac{\partial g\left(x,t\right)\ \ \ }{\partial t}+\frac{\partial g\left(x,t\right)\ \ \ }{\partial x}\ =g\left(x,t\right)\frac{S_1 '(x)}{S_1(x)}\ ,                                            
\end{equation} 
\begin{equation} \label{GrindEQ__13_} 
\frac{\partial m\left(x,t\right)\ \ \ }{\partial t}+\frac{\partial m\left(x,t\right)\ \ \ }{\partial x}\ =m\left(x,t\right)\frac{S_2 '(x)}{S_2(x)},                                          
\end{equation} 
with the boundary conditions
\[g\left(0,x\right)=r\int^{\infty }_0{g\left(x,t\right)}b\left(x,t\right)dx,\] 
\[m\left(0,x\right)=\left(1-r\right)\int^{\infty }_0{g\left(x,t\right)}b\left(x,t\right)dx.\] 
\end{theorem}

\begin{proof}
In equality\textbf{\textit{ }}\eqref{GrindEQ__6_} the exponential is expanded in series up to \textit{o}($\varepsilon$${}^{2}$). In view of (7-10), we make the passage to the limit for $\varepsilon$$\mathrm{\to}$0. We obtain that the functional \textit{F}${}_{1}$(\textit{w},\textit{t}) is the solution of a first-order linear homogeneous partial differential equation
\[\frac{\partial F_1(w_1,w_2,t)}{\partial t}=\int^{\infty }_0{\frac{\partial F_1(w_1,w_2,t)}{\partial w_1(x)}}\Bigg\{rw_1\left(0\right)b\left(x,t\right)\Bigg.+\] 
\begin{equation} \label{GrindEQ__14_} 
+\left(1-r\right)w_2\left(0\right)b\left(x,t\right)+w_1\left(x\right)\frac{S_1 '\left(x\right)}{S_1\left(x\right)}+\Bigg.{w_1} '\left(x\right)\Bigg\}dx+                         
\end{equation} 
\[+\int^{\infty }_0{\frac{\partial F_1(w_1,w_2,t)}{\partial w_2(x)}}\left\{w_2\left(x\right)\frac{S_2 '\left(x\right)}{S_2\left(x\right)}+{w_2} '\left(x\right)\right\}dx.\] 
 The solution of this equation is found in the form of the characteristic functional
\begin{equation} \label{GrindEQ__15_} 
F_1\left(w_1,w_2,t\right)=exp\left\{j\int^{\infty }_0{\left(w_1\left(x\right)g\left(x,t\right)+w_2\left(x\right)m\left(x,t\right)\right)dx}\right\}, 
\end{equation} 
where function\textit{ g}(\textit{x},\textit{t}) is the mean stochastic age-specific density $\xi$(\textit{x},\textit{t}) of the asymptotic number of application of the first type, serviced in the system at the moment \textit{t}, and function\textit{ m}(\textit{x},\textit{t}) is the mean stochastic age-specific density $\eta$(\textit{x},\textit{t}) of the asymptotic number of application of the second type, serviced in the system at the moment \textit{t},  $j=\sqrt{-1}\ $-- the imaginary unit. Substitute the solution \eqref{GrindEQ__15_} in equation \eqref{GrindEQ__14_}
\[\int^{\infty }_0{w_1\left(x\right)\frac{\partial g\left(x,t\right)}{\partial t}}dx+\int^{\infty }_0{w_2\left(x\right)\frac{\partial m\left(x,t\right)}{\partial t}dx}=r\int^{\infty }_0{g\left(x,t\right)w_1\left(0\right)}b\left(x,t\right)dx+\] 
\[+\int^{\infty }_0{g\left(x,t\right)w_1\left(x\right)}\frac{S_1 '\left(x\right)}{S_1\left(x\right)}dx+(1-r)\int^{\infty }_0{g\left(x,t\right)w_2\left(0\right)}b\left(x,t\right)dx+\int^{\infty }_0{m\left(x,t\right)w_2\left(x\right)}\frac{S_2 '\left(x\right)}{S_2\left(x\right)}dx-\] 
\[-w_1\left(0\right)g\left(x,t\right)-w_2\left(0\right)m\left(x,t\right).\] 

Equating terms for the same \textit{w${}_{1}$}(0), \textit{w${}_{2}$}(0), and \textit{w${}_{1}$}(\textit{x}), \textit{w${}_{2}$}(\textit{x}), we obtain the equation that coincides with \eqref{GrindEQ__14_} and the corresponding boundary condition. The theorem is proved.
\end{proof} 

By virtue of \eqref{GrindEQ__15_}, we can write the asymptotic equality for $\varepsilon$$\mathrm{\to}$0 
\[F\left(y,t\right)\mathrm{=}F_{\mathrm{1}}\left(w_{\mathrm{1}},w_{\mathrm{2}},t\right)\mathrm{+}O\left(\varepsilon \right)\mathrm{=}exp\left\{jN\int^{\mathrm{\infty }}_0{\mathrm{(u}\left(x\right)g\left(x,t\right)\mathrm{+z}\left(x\right)m\left(x,t\right)\mathrm{)}dx}\right\}\mathrm{+}O\left(\varepsilon \right).\] 

\begin{definition}
A function
\[\widetilde{F_1}(y,t)=exp\left\{jN\int^{\mathrm{\infty }}_0{\mathrm{(u}\left(x\right)g\left(x,t\right)\mathrm{+z}\left(x\right)m\left(x,t\right)\mathrm{)}dx}\right\}\] 
is called the first order asymptotic of the characteristic functional\textit{ F}(y,\textit{t}).
\end{definition}

In terms of demography, equation (12-13) together with the boundary conditions determines the function \textit{Ng}(\textit{x},\textit{t}) and \textit{Nm}(\textit{x},\textit{t}). \textit{Ng}(\textit{x},\textit{t}) and \textit{Nm}(\textit{x},\textit{t}) are an average value of the stochastic population density of women and men respectively.

The analytic solution of the partial differential equations \eqref{GrindEQ__12_} and \eqref{GrindEQ__13_} are founded by the composition and solution of the system, determined the characteristics. We write down the solution of \eqref{GrindEQ__12_}
\[g\left(x,t\right)\ =S_1\left(x\right)\frac{g\left(x-t,0\right)}{S_1\left(x-t\right)}\ ,for\ x\ge t,\ x\ge 0,\ t\ge 0,\] 
where \textit{g}(\textit{x},0) is the known initial conditions,
\[g\left(x,t\right)=S_1\left(x\right)\varphi \left(t-x\right),\ for\ x<t,\ x\ge 0,\ t\ge 0,\] 
where $\varphi$(\textit{t}) is the solution of Voltaire integral equation of the second kind 
\[\varphi \left(t\right)=r\int^t_0{S_1\left(y\right)\varphi \left(t-y\right)b\left(y,t\right)dy+r\int^{\infty }_t{S_1\left(y\right)\frac{g\left(y-t,0\right)}{S_1\left(y-t\right)}b\left(y,t\right)dy}}.\]

Similarly, we write the solution of \eqref{GrindEQ__13_}
\[m\left(x,t\right)\ =S_2\left(x\right)\frac{m\left(x-t,0\right)}{S_2\left(x-t\right)}\ ,for\ x\ge t,\ x\ge 0,\ t\ge 0,\] 
where \textit{m}(\textit{x},0) is the known initial conditions,
\[m\left(x,t\right)=S_2\left(x\right)\phi \left(t-x\right),\ for\ x<t,\ x\ge 0,\ t\ge 0,\] 
where $\phi $(\textit{t}) is the solution of Voltaire integral equation of the second kind
\[\phi \left(t\right)=(1-r)\int^t_0{S_1\left(y\right)\varphi \left(t-y\right)b\left(y,t\right)dy+}(1-r)\int^{\infty }_t{S_1\left(y\right)\frac{g\left(y-t,0\right)}{S_1\left(y-t\right)}b\left(y,t\right)dy}.\] 

\subsubsection{The second order asymptotic}

We note that the first-order asymptotics determines only the average values of the stochastic densities $\xi$(\textit{x},\textit{t}) and $\eta$(\textit{x},\textit{t}). Therefore, naturally, the need arises to find second-order asymptotics, which allows one to obtain more detailed characteristics.

In order to find the second order asymptotic, we carry out some transformations in the basic equation \eqref{GrindEQ__5_}.

 In equation \eqref{GrindEQ__11_} by virtue of substitution \textit{u}=$\varepsilon$\textit{w${}_{1}$}, \textit{z}=$\varepsilon$\textit{w${}_{2}$}, where $\varepsilon$=1/\textit{N,} we return to the function \textit{u(x)} and \textit{z(x)}, obtain
\[F_1\left(w_1,w_2,t\right)=exp\left\{j\frac{1}{\varepsilon }\int^{\infty }_0{\left({\varepsilon w}_1\left(x\right)g\left(x,t\right)+{\varepsilon w}_2\left(x\right)m\left(x,t\right)\right)dx}\right\}=\] 
\[=exp\left\{jN\int^{\infty }_0{(u\left(x\right)g\left(x,t\right)++z\left(x\right)m\left(x,t\right))dx}\right\}.\] 
In equation \eqref{GrindEQ__5_} we replace
\begin{equation} \label{GrindEQ__16_} 
F\left(y,t\right)\mathrm{=}H_2\left(\mathrm{y,}t\right)exp\left\{jN\int^{\mathrm{\infty }}_0{\mathrm{(u}\left(x\right)g\left(x,t\right)\mathrm{+z}\left(x\right)m\left(x,t\right)\mathrm{)}dx}\right\},                 
\end{equation} 
and obtain
\[\frac{\partial H_2\left(y,t\right)}{\partial t}=j\int^{\infty }_0{\frac{\partial H_2\left(y,t\right)}{\partial u\left(x\right)}}\Bigg\{\left(1-re^{ju\left(0\right)}\right)b\left(x,t\right)\Bigg.-\] 
\[-\left(1-r\right)e^{jz\left(0\right)}b\left(x,t\right)+\left(e^{-ju\left(x\right)}-1\right)\frac{S_1 '\left(x\right)}{S_1\left(x\right)}-\Bigg.ju '\left(x\right)\Bigg\}dx -\] 
\[- NH_2\left(y,t\right)\{j\int^{\infty }_0{u\left(x\right)\frac{\partial g\left(x,t\right)}{\partial t}dx+\int^{\infty }_0{g\left(x,t\right)}}\Bigg\{\left(1-re^{ju\left(0\right)}\right)b\left(x,t\right)\Bigg.\ -\] 
\begin{equation} \label{GrindEQ__17_} 
-\left(1-r\right)e^{jz\left(0\right)}b\left(x,t\right)+\left(e^{-ju\left(x\right)}-1\right)\frac{S_1 '\left(x\right)}{S_1\left(x\right)}-\Bigg.ju '\left(x\right)\Bigg\}\Bigg.dx\Bigg\} +                   
\end{equation} 
\[+j\int^{\infty }_0{\frac{\partial H_2(y,t)}{\partial z(x)}}\Bigg\{\left(e^{-jz\left(x\right)}-1\right)\frac{S_2 '\left(x\right)}{S_2\left(x\right)}-jz '\left(x\right)\Bigg\}dx -\] 
\[-NH_2\left(y,t\right)\{j\int^{\infty }_0{z\left(x\right)\frac{\partial m\left(x,t\right)}{\partial t}dx+\int^{\infty }_0{m\left(x,t\right)}}\{\left(e^{-jz\left(x\right)}-1\right)\frac{S_2 '\left(x\right)}{S_2\left(x\right)} -\] 
\[-\Bigg.jz '\left(x\right)\Bigg\}\Bigg.dx\Bigg\}.\] 
From a prior designation and \eqref{GrindEQ__16_} follow, that \textit{H}${}_{2}$(\textit{y},\textit{t}) is the characteristic functional for the quantity ($\xi$(\textit{x},\textit{t})+$\eta$(\textit{x},\textit{t})) --\textit{N(g(x,t)+m}(\textit{x},\textit{t})). The mathematical expectation of $\mathrm{\{}$($\xi$(\textit{x},\textit{t})+$\eta$(\textit{x},\textit{t})) -- \textit{N(g(x,t)+m}(\textit{x},\textit{t})$\mathrm{\}}$ equal to zero. 

Denote $\varepsilon$${}^{2}$=1/\textit{N} and make replacements in the equation \eqref{GrindEQ__17_} 
\begin{equation} \label{GrindEQ__18_} 
u\left(x\right)=\varepsilon w_1\left(x\right),z\left(x\right)=\varepsilon w_2\left(x\right),H_2\left(y,t\right)=F_2\left(w_1,w_2,t,\varepsilon \right), 
\end{equation} 
obtain the equality 
\[\frac{\partial F_2(w_1,w_2,t,\varepsilon )}{\partial t}=j\frac{1}{\varepsilon }\int^{\infty }_0{\frac{\partial F_2(w_1,w_2,t,\varepsilon )}{\partial w_1(x)}}\Bigg\{\left(1-re^{j{\varepsilon w}_1\left(0\right)}\right)b\left(x,t\right)\Bigg. -\] 
\[-\left(1-r\right)e^{j\varepsilon w_2\left(0\right)}b\left(x,t\right)+\left(e^{-j\varepsilon w_1\left(x\right)}-1\right)\frac{S_1 '\left(x\right)}{S_1\left(x\right)}-\Bigg.j\varepsilon {w_1} '\left(x\right)\Bigg\}dx-\] 
\[-F_2(w_1,w_2,t,\varepsilon )\frac{1}{{\varepsilon }^2}\Bigg\{j\varepsilon \int^{\infty }_0{w_1\left(x\right)\frac{\partial g\left(x,t\right)}{\partial t}dx+\int^{\infty }_0{g\left(x,t\right)}}\Bigg. \cdot \] 
\begin{equation} \label{GrindEQ__19_} 
\cdot \Bigg\{\left(1-re^{j\varepsilon w_1\left(0\right)}\right)b\left(x,t\right)\Bigg.-\left(1-r\right)e^{j{\varepsilon w}_2\left(0\right)}b\left(x,t\right)+\left(e^{-j\varepsilon w_1\left(x\right)}-1\right)\frac{S_1 '\left(x\right)}{S_1\left(x\right)} -     
\end{equation} 
\[-\Bigg.{j\varepsilon w_1} '\left(x\right)\Bigg\}\Bigg.dx\Bigg\}+j\frac{1}{\varepsilon }\int^{\infty }_0{\frac{\partial F_2\left(w_1,w_2,t,\varepsilon \right)}{\partial w_2\left(x\right)}}\Bigg\{\left(e^{-j\varepsilon w_2\left(x\right)}-1\right)\frac{S_2 '\left(x\right)}{S_2\left(x\right)} \Bigg.-\] 
\[\Bigg.-j\varepsilon {w_2} '\left(x\right)\Bigg\}-F_2\left(w_1,w_2,t,\varepsilon \right)\frac{1}{{\varepsilon }^2}\Bigg\{j\int^{\infty }_0{w_2\left(x\right)\frac{\partial m\left(x,t\right)}{\partial t}dx+\int^{\infty }_0{m\left(x,t\right)}} \Bigg. \cdot\] 
\[ \cdot \Bigg\{\left(e^{-j{\varepsilon w}_2\left(x\right)}-1\right)\frac{S_2 '\left(x\right)}{S_2\left(x\right)}-\Bigg.{j{\varepsilon w}_2} '\left(x\right)\Bigg\}\Bigg.dx\Bigg\}.\Bigg.\] 

This equality \eqref{GrindEQ__19_} is a singularly perturbed equation. We proceed similarly to the first order asymptotic. We consider a subclass of solutions \textit{H}${}_{2}$(\textit{y,t}) of equation \eqref{GrindEQ__17_}, for which, by \eqref{GrindEQ__18_}, the function \textit{F}${}_{2}$(\textit{w${}_{1}$},\textit{w${}_{2}$,t,}$\varepsilon$) has the following properties.

There is a finite limit at $\varepsilon$$\mathrm{\to}$0 for \textit{F}${}_{2}$(\textit{w${}_{1}$},\textit{w${}_{2}$,t,}$\varepsilon$)
\begin{equation} \label{GrindEQ__20_} 
{\mathop{\mathrm{lim}}_{\varepsilon \to 0} F_2\left(w_1,w_2,t,\varepsilon \right)=F_2\left(w_1,w_2,t\right)\ },                                     
\end{equation} 
and its derivative with respect \textit{t}, $w_1\left(x\right)$ and $w_2\left(x\right)$
\begin{equation} \label{GrindEQ__21_} 
{\mathop{\mathrm{lim}}_{\varepsilon \to 0} \frac{{\partial F}_2\left(w_1,w_2,t,\varepsilon \right)}{\partial t}=\frac{F_2\left(w_1,w_2,t\right)}{\partial t}\ },                                             
\end{equation} 
\begin{equation} \label{GrindEQ__22_} 
{\mathop{\mathrm{lim}}_{\varepsilon \to 0} \frac{{\partial F}_2\left(w_1,w_2,t,\varepsilon \right)}{\partial w_1(x)}=\frac{F_2\left(w_1,w_2,t\right)}{\partial w_1(x)}\ },                                             
\end{equation} 
\begin{equation} \label{GrindEQ__23_} 
{\mathop{\mathrm{lim}}_{\varepsilon \to 0} \frac{{\partial F}_2\left(w_1,w_2,t,\varepsilon \right)}{\partial w_2(x)}=\frac{F_2\left(w_1,w_2,t\right)}{\partial w_2(x)}\ }.                                             
\end{equation} 
We state the following theorem.

\begin{theorem}
\label{thm2}
For $\varepsilon$$\mathrm{\to}$0 the solution of the task \eqref{GrindEQ__19_} has the form
\[F_2\left(w_1,w_2,t\right)=exp\Bigg\{-\frac{1}{2}\Bigg[\mathop{\int\!\!\!\!\int}\nolimits^{\infty }_0{w_1\left(y\right)w_1}\left(z\right)R_{11}\left(y,z,t\right)dydz\Bigg.\Bigg.+\] 
\[+2\mathop{\int\!\!\!\!\int}\nolimits^{\infty }_0{w_1\left(y\right)w_2}\left(z\right)R_{12}\left(y,z,t\right)dydz+\Bigg.\Bigg.\mathop{\int\!\!\!\!\int}\nolimits^{\infty }_0{w_2\left(y\right)w_2}\left(z\right)R_{22}\left(y,z,t\right)dydz\Bigg]\Bigg\},\] 
were \textit{R}${}_{11}$(\textit{y},\textit{z},\textit{t}), \textit{R}${}_{1}$${}_{2}$(\textit{y},\textit{z},\textit{t})  and \textit{R}${}_{22}$(\textit{y},\textit{z},\textit{t}) are cross-correlation functions, 
\end{theorem}

in order to study the variance of the stochastic density of applications of zero-age and age x, as well as the correlation dependences of the stochastic density of applications of zero-age and age x, for all y$\mathrm{\ge}$0, z$\mathrm{\ge}$0 we write the cross-correlation functions in the form symmetric with respect to the arguments y and z:
\[R_{11}\left(y,z,t\right)={\sigma }^2_1\left(t\right)\delta \left(y\right)\delta \left(z\right)+r_1\left(y,t\right)\delta \left(z\right)+r_1\left(z,t\right)\delta \left(y\right)+{\sigma }_1\left(y,t\right){\sigma }_1\left(z,t\right)\delta \left(y-z\right),\] 
\[R_{22}\left(y,z,t\right)={\sigma }^2_2\left(t\right)\delta \left(y\right)\delta \left(z\right)+r_2\left(y,t\right)\delta \left(z\right)+r_2\left(z,t\right)\delta \left(y\right)+{\sigma }_2\left(y,t\right){\sigma }_2\left(z,t\right)\delta \left(y-z\right),\] 
\[R_{12}\left(y,z,t\right)={\sigma }^2_{12}\left(t\right)\delta \left(y\right)\delta \left(z\right)+r_{12}\left(y,t\right)\delta \left(z\right)+r_{12}\left(z,t\right)\delta \left(y\right)+{\sigma }_{12}\left(y,t\right){\sigma }_{12}\left(z,t\right)\delta \left(y-z\right),\] 
where $\delta \left(.\right)$ -- $\delta$-Dirac function. It follows from the form of cross-correlation functions \textit{R}${}_{11}$(\textit{y},\textit{z},\textit{t}), \textit{R}${}_{1}$${}_{2}$(\textit{y},\textit{z},\textit{t})  and \textit{R}${}_{22}$(\textit{y},\textit{z},\textit{t})  that the normalized numbers of applications of ages \textit{y}$\mathrm{>}$0 and \textit{z}$\mathrm{>}$0 for \textit{z}$\mathrm{\neq}$\textit{y} are stochastically independent, while the numbers of applications of zero age and age \textit{x} are dependent for all values of ages \textit{x}$\mathrm{\ge}$0 for \textit{b}(\textit{x},\textit{t})$\mathrm{>}$ 0, and
\[{\sigma }^2_1\left(0\right)=0,\ r_1\left(x,0\right)=0,r_1\left(0,t\right)=0,{\sigma }_1\left(0,t\right)=0,{\sigma }_1\left(x,0\right)=0,\] 
\[{\sigma }^2_2\left(0\right)=0,\ r_2\left(x,0\right)=0,r_2\left(0,t\right)=0,{\sigma }_2\left(0,t\right)=0,{\sigma }_2\left(x,0\right)=0,\] 
\[{\sigma }^2_{12}\left(0\right)=0,\ r_{12}\left(x,0\right)=0,r_{12}\left(0,t\right)=0,{\sigma }_{12}\left(0,t\right)=0,{\sigma }_{12}\left(x,0\right)=0\] 
are the initial and boundary conditions, and the summands of the cross-correlation functions are found from the system of partial differential equations
\[\frac{{\partial \sigma }^2_1\left(t\right)}{\partial t}=2S_1 '\left(0\right){\sigma }^2_1\left(t\right)+2r\int^{\infty }_0{r_1\left(x,t\right)}b\left(x,t\right)dx+rg\left(0,t\right),\] 
\[\frac{{\partial \sigma }^2_2\left(t\right)}{\partial t}=2S_2 '\left(0\right){\sigma }^2_2\left(t\right)+2(1-r)\int^{\infty }_0{r_{12}\left(x,t\right)}b\left(x,t\right)dx+(1-r)g\left(0,t\right),\] 
\[\frac{{\partial \sigma }^2_1\left(x,t\right)}{\partial t}+\frac{{\partial \sigma }^2_1\left(x,t\right)}{\partial x}=\left(2{\sigma }^2_1\left(x,t\right)-g\left(x,t\right)\right)\frac{S_1 '\left(x\right)}{S_1\left(x\right)},\] 
\[\frac{{\partial \sigma }^2_2\left(x,t\right)}{\partial t}+\frac{{\partial \sigma }^2_2\left(x,t\right)}{\partial x}=\left(2{\sigma }^2_2\left(x,t\right)-m\left(x,t\right)\right)\frac{S_2 '\left(x\right)}{S_2\left(x\right)},\] 
\[\frac{{\partial r}_1\left(x,t\right)}{\partial t}+\frac{{\partial r}_1\left(x,t\right)}{\partial x}=r_1\left(x,t\right)\left(S_1 '\left(0\right)+\frac{S_1 '\left(x\right)}{S_1\left(x\right)}\right)+r{\sigma }^2_1\left(x,t\right)b(x,t),\] 
\begin{equation} \label{GrindEQ__24_} 
\frac{{\partial r}_2\left(x,t\right)}{\partial t}+\frac{{\partial r}_2\left(x,t\right)}{\partial x}=r_2\left(x,t\right)\left(S_2 '\left(0\right)+\frac{S_2 '\left(x\right)}{S_2\left(x\right)}\right)+(1-r){\sigma }^2_{12}\left(x,t\right)b(x,t),          
\end{equation} 
\[\frac{{\partial \sigma }^2_{12}\left(t\right)}{\partial t}=(S_1 '\left(0\right)+S_2 '\left(0\right)){\sigma }^2_{12}\left(t\right)+(1-r)\int^{\infty }_0{r_1\left(x,t\right)}b\left(x,t\right)dx+r\int^{\infty }_0{r_{12}\left(x,t\right)}b\left(x,t\right)dx,\] 
\[\frac{{\partial \sigma }^2_{12}\left(x,t\right)}{\partial t}+\frac{{\partial \sigma }^2_{12}\left(x,t\right)}{\partial x}={\sigma }^2_{12}\left(x,t\right)\left(\frac{S_1 '\left(x\right)}{S_1\left(x\right)}+\frac{S_2 /\left(x\right)}{S_2\left(x\right)}\right),\] 
\[\frac{{\partial r}_{12}\left(x,t\right)}{\partial t}+\frac{{\partial r}_{12}\left(x,t\right)}{\partial x}=r_{12}\left(x,t\right)\left(S_2 '\left(0\right)+\frac{S_1 '\left(x\right)}{S_1\left(x\right)}\right)+(1-r){\sigma }^2_1\left(x,t\right)b(x,t),\] 
\[\frac{{\partial r}_{12}\left(x,t\right)}{\partial t}+\frac{{\partial r}_{12}\left(x,t\right)}{\partial x}=r_{12}\left(x,t\right)\left(S_1 '\left(0\right)+\frac{S_2 '\left(x\right)}{S_2\left(x\right)}\right)+r{\sigma }^2_{12}\left(x,t\right)b\left(x,t\right).\] 
The proof of Theorem~\ref{thm2}  is completely analogous to the proof of Theorem~\ref{thm1}. Due to its bulkiness, we skip it. 

 By virtue of \eqref{GrindEQ__18_} and also equality \textit{F}${}_{2}$(\textit{w${}_{1}$},\textit{w${}_{2}$,t}) from Theorem~\ref{thm2}, we can write the asymptotic equality for $\varepsilon$$\mathrm{\to}$0 
\[F_2\left(y,t\right)=F_2\left(w_1,w_2,t\right)+o(\varepsilon )=exp\left\{-\frac{1}{2}\left[\mathop{\int\!\!\!\!\int}\nolimits^{\infty }_0{w_1\left(y\right)w_1}\left(z\right)R_{11}\left(y,z,t\right)dydz\right.\right.+\] 
\[+2\mathop{\int\!\!\!\!\int}\nolimits^{\infty }_0{w_1\left(y\right)w_2}\left(z\right)R_{12}\left(y,z,t\right)dydz\left.\left.+\mathop{\int\!\!\!\!\int}\nolimits^{\infty }_0{w_2\left(y\right)w_2}\left(z\right)R_{22}\left(y,z,t\right)dydz\right]\right\}+o\left(\varepsilon \right).\] 
Therefore, we can write
\[F\left(y,t\right)\mathrm{=}H_2\left(\mathrm{y,}t\right)exp\left\{jN\int^{\mathrm{\infty }}_0{\mathrm{(u}\left(x\right)g\left(x,t\right)\mathrm{+z}\left(x\right)m\left(x,t\right)\mathrm{)}dx}\right\}=\] 
\[=exp\left\{jN\int^{\mathrm{\infty }}_0{\left(\mathrm{u}\left(x\right)g\left(x,t\right)\mathrm{+z}\left(x\right)m\left(x,t\right)\right)dx}\right.-\] 
\[-\frac{1}{2}\left[\mathop{\int\!\!\!\!\int}\nolimits^{\infty }_0{w_1\left(y\right)w_1}\left(z\right)R_{11}\left(y,z,t\right)dydz+2\mathop{\int\!\!\!\!\int}\nolimits^{\infty }_0{w_1\left(y\right)w_2}\left(z\right)R_{12}\left(y,z,t\right)dydz\right.+\] 
\[\left.\left.+\mathop{\int\!\!\!\!\int}\nolimits^{\infty }_0{w_2\left(y\right)w_2}\left(z\right)R_{22}\left(y,z,t\right)dydz\right]\right\}+o\left(\varepsilon \right).\] 

\begin{definition}
A function
\[{\tilde{F}}_2(y,t)=exp\left\{jN\int^{\mathrm{\infty }}_0{\left(\mathrm{u}\left(x\right)g\left(x,t\right)\mathrm{+z}\left(x\right)m\left(x,t\right)\right)dx}\right. -\] 
\[-\frac{1}{2}\left[\mathop{\int\!\!\!\!\int}\nolimits^{\infty }_0{w_1\left(y\right)w_1}\left(z\right)R_{11}\left(y,z,t\right)dydz+2\mathop{\int\!\!\!\!\int}\nolimits^{\infty }_0{w_1\left(y\right)w_2}\left(z\right)R_{12}\left(y,z,t\right)dydz\right.+\] 
\[\left.\left.+\mathop{\int\!\!\!\!\int}\nolimits^{\infty }_0{w_2\left(y\right)w_2}\left(z\right)R_{22}\left(y,z,t\right)dydz\right]\right\}\] 
is called the second order asymptotic of the characteristic functional\textit{ F}(y,\textit{t}).\textit{}
\end{definition}

We write down the solution of the system \eqref{GrindEQ__24_} of partial differential equations of Theorem~\ref{thm2}
\[{\sigma }^2_1\left(t\right)=e^{2S_1 '\left(0\right)t\ }\left\{r\int^t_0{g\left(0,y\right)e^{-2S_1 '\left(0\right)y\ }dy+}2r\int^t_0{e^{-2S '_1\left(0\right)y\ }\int^{\infty }_0{r_1(x,y)}b\left(x,y\right)dxdy}\right\}, for \:{x}\ge 0, t\ge 0,\] 
\[{\sigma }^2_2\left(t\right)=e^{2S '_2\left(0\right)t\ }\left\{(1-r)\int^t_0{g\left(0,y\right)e^{-2S '_2\left(0\right)y\ }dy+}2(1-r)\int^t_0{e^{-2S '_2\left(0\right)y\ }\int^{\infty }_0{r_{12}(x,y)}b\left(x,y\right)dxdy}\right\},\] 
\[for \:{x}\ge 0, t\ge 0,\] 
\[r_1\left(x,t\right)={re}^{S '_1\left(0\right)x\ }S_1(x)\int^x_0{e^{-S '_1\left(0\right)y\ }S^{-1}_1\left(y\right){\sigma }^2_1\left(y,t\right)}b\left(y,t\right)dy, for \:{x}\ge 0,t\ge 0,\] 
\[r_2\left(x,t\right)=0, for \:{x}\ge 0,t\ge 0,\] 
\[{\sigma }^2_1\left(x,t\right)=S^2_1\left(x\right)\left\{\int^{x-t}_0{g\left(y,0\right)S '_1\left(y\right)S^{-3}_1\left(y\right)dy-\int^x_0{g\left(y,t\right)S '_1\left(y\right)S^{-3}_1\left(y\right)dy}}\right\},\] 
\[for \:{x}\ge t, x\ge 0, t\ge 0,\] 
\[{\sigma }^2_1\left(x,t\right)=-S^2_1\left(x\right)\int^x_0{g\left(y,t\right)S '_1\left(y\right)S^{-3}_1\left(y\right)dy}, for \:{x}\ge t,x\ge 0, t\ge 0.\] 
\begin{equation} \label{GrindEQ__25_} 
{\sigma }^2_2\left(x,t\right)=S^2_2\left(x\right)\left\{\int^{x-t}_0{m\left(y,0\right)S '_2\left(y\right)S^{-3}_2\left(y\right)dy-\int^x_0{m\left(y,t\right)S '_2\left(y\right)S^{-3}_2\left(y\right)dy}}\right\}, 
\end{equation} 
\[for \:{x}\ge t, x\ge 0, t\ge 0,\] 
\[{\sigma }^2_2\left(x,t\right)=-S^2_2\left(x\right)\int^x_0{m\left(y,t\right)S '_2\left(y\right)S^{-3}_2\left(y\right)dy}, for \:{x\ge t, x\ge 0, t\ge 0},\] 
\[{\sigma }^2_{12}\left(t\right)=e^{(S '_1\left(0\right)+S '_2\left(0\right))t}\left\{\left(1-r\right)\int^t_0{e^{-\left(S '_1\left(0\right)+S '_2\left(0\right)\right)y}}\int^{\infty }_0{r_1\left(x,y\right)}b\left(x,y\right)dydx+\right.\] 
\[+r\left.\int^t_0{e^{-(S '_1\left(0\right)+S '_2\left(0\right))y}}\int^{\infty }_0{r_{12}\left(x,t\right)}b\left(x,t\right)dydx\right\} for \:{x\ge 0, t\ge 0},\] 
\[{\sigma }^2_{12}\left(x,t\right)=0, for \:{x\ge 0, t\ge 0},\] 
\[r_{12}\left(x,t\right)=0,25\ (1-{r)e}^{{(S} '_1\left(0\right)+S '_2\left(0\right))x\ }S_1\left(x\right)S_2\left(x\right) \cdot \] 
\[\cdot \int^x_0{e^{-(S '_1\left(0\right)+S '_2\left(0\right))y\ }S {-1}_1\left(y\right)\ S^{-1}_2\left(y\right){\sigma }^2_1\left(y,t\right)}b\left(y,t\right)dy, for \:{x\ge 0, t\ge 0}.\] 

The main result is the analytical solutions \eqref{GrindEQ__12_}, \eqref{GrindEQ__13_} and \eqref{GrindEQ__25_} which determine the average characteristics and second moments for the stochastic densities of requests $\xi$(\textit{x},\textit{t}) and $\eta$(\textit{x},\textit{t}), as well as the proof that the asymptotic distribution is Gaussian. In this way the research of the autonomous non-Markov queuing system with an unlimited number of devices is completed.

\section{Conclusion}

The article proposes a stochastic model of demographic growth in the form of an autonomous non-Markov queuing system with an unlimited number of devices and two types of applications. Its research was carried out using the virtual phase method, which consists in approximating the service time by the sum of a random number of independent identically exponentially distributed random variables and a limit transition with an unlimited increase in the number of phases and a proportional decrease in the duration of each phase. Using the modified method of asymptotic analysis of a stochastic age-specific density for a number of applications served in the system at time t we were found the main probabilistic characteristics of the number of served applications in the system and was proofed that their asymptotic distribution is Gaussian. To apply the obtained results to the study of demographic processes, it is necessary to choose the explicit form of unknown functions (survival function and a fertility function) and set the parameter values.

Created mathematical model has quite wide possibilities for generalization and modification. This mathematical model of human population growth can be applied to predict the demographic situation in any country and in the world as a whole, including for forecasting the population size taking into account the age structure, marital structure and social status, for forecasting migration processes.

\acks It is a pleasure to thank Prof. A.A. Nazarov for helpful advice. We also wish to thank the anonymous referees whose comments greatly helped to improve the paper.

\end{document}